\documentclass[plain]{amsart}
\usepackage[toc,page]{appendix}
\usepackage{amssymb,graphicx,epsfig,amsxtra, amsmath}
\usepackage[]{graphics}
\usepackage{amscd} 
\usepackage{xypic}
\usepackage[mathscr]{eucal}

\usepackage{xcolor}
\usepackage{a4}
\usepackage[normalem]{ulem}
\usepackage{paralist}
\usepackage{comment}
\input{xy}
\xyoption{all}
\newtheorem{theorem}{\textbf{Theorem}}[section]

\newtheorem{lemma}[theorem]{\textbf{Lemma}}
\newtheorem{corollary}[theorem]{\textbf{Corollary}}

\theoremstyle{definition}

\newtheorem{remark}[theorem]{Remark}

\newtheorem{definition-remark}{Definition-Remark}

\def\ag{\`a}

\def\geq{\geqslant}
\def\leq{\leqslant}

\newcommand{\sM}{{\mathcal M}}

\newcommand{\sT}{{\mathcal T}}

\newcommand{\PP}{\ensuremath{\mathbb{P}}}

\newcommand{\CC}{\ensuremath{\mathbb{C}}}

\newcommand{\NN}{\ensuremath{\mathbb{N}}}
\newcommand{\hol}{\ensuremath{\mathcal{O}}}

\newcommand\de{\delta}

\begin{document}

\title[A remark on  isolated complex hypersurface singularities]
{A remark on  isolated complex hypersurface singularities}


\author{ Fabrizio Catanese }
\address{Mathematisches Institut\\
Fakult\"at f\"ur Mathematik, Physik und Informatik\\
Universit\"at Bayreuth
}
\email{fabrizio.catanese@uni-bayreuth.de}

\author
{Ciro Ciliberto}
\address{Dipartimento di Matematica\\
Universit\`a di Roma ``Tor Vergata''\\
 Via della Ricerca Scientifica\\00133 Roma, Italy.}
\email{cilibert@mat.uniroma2.it}

\author{Concettina Galati }
\address{Dipartimento di Matematica\\
Universit\ag\, della Calabria\\
via P. Bucci, cubo 31B\\
87036 Arcavacata di Rende (CS), Italy. }
\email{concettina.galati@unical.it }



\date{}

\dedicatory{}

\commby{}


\begin{abstract}  This is now an expository note about the following classical problem.  Let $(X, \bf 0)$ be the germ of a hypersurface in $(\mathbb C^n,\bf 0)$ with an ordinary singularity of multiplicity $m$ at the origin $\bf 0$.  A natural question  to ask is whether $X$ and its tangent cone at the origin are analytically isomorphic. The answer is negative  in general, in view of a theorem of Kioji Saito.  

 However 
 there is an integer $D(n,m)>m$ such that, given a \emph{regular} homogeneous polynomial $f(x_1,\ldots, x_n)$ of degree $m$
 (this means that $\{ f=0\}$   is  a  smooth  hypersurface in $\PP^{n-1}$) then,  for all $d\geq D(n,m)$,  any convergent  power series of the form $g=f+ o(d)$ (here, as usual,  $o(d)$ stays for a power series  of order at least $d$), 
 defines a germ  $\{ g=0\}$ which is analytically equivalent to the germ $\{ f=0\}$. In this note we  compute $D(n,m)$
 explicitly as $n(m-2)+1$.  We also give an extension to the case in which $f$ is a quasihomogeneous polynomial.
  It was pointed out that the value of $D(n,m)$ was already known by \cite[Exercise 7.31]{D}. 
\end{abstract}


\maketitle

\tableofcontents

\section{Introduction} 
Let $(X,p)$ be the  germ of an isolated   hypersurface singularity in $\mathbb C^n$, such that the tangent cone $C_{X,p}$ of $X$ at $p$ is the cone over a smooth hypersurface of degree $m\geq 2$ in $\mathbb P^{n-1}$, where $n\geq 2$. Then we shall say that $p$ is an \emph{ordinary singularity of multiplicity $m$} or an \emph{ordinary $m$--tuple point} of $X$.  A natural question  to ask  is  whether $(X,p)$ and $(C_{X,p},p)$ are analytically isomorphic. It is well known that this is always the case if $m=2$.  The  answer is also affirmative  if $m=n=3$  \cite [Prop. 3.3]{CG},  and  if $3\leq m\leq 4$ and $n=2$ (see \cite {F}). 

Despite this, the answer is, as we will see, negative in general. However, by the \emph{finite determinacy theorem} (see \cite [Thm. 1.2.23, 2nd ed.] {GLS}   and also  \cite[Ch. 6]{D}), the following holds. 

Say that a homogeneous polynomial   $f(x_1,\ldots, x_n)\in \mathbb C[x_1,\ldots, x_n]$
of degree $m$  is \emph{regular} if  the hypersurface of $\mathbb P^{n-1}$ with equation $f=0$ is smooth. 

Then there is an integer $D(n,m)>m$ such that, for all such regular polynomials $f$ and for all $d\geq D(n,m)$, the following happens: if we consider a convergent  power series of the form $g=f+ o(d)$, where $o(d)$ stays for a power series  of order at least $d$, then the germ at the origin defined by $g=0$ is analytically equivalent to the germ defined by $f=0$. 

In this note we will compute $D(n,m)$ for $ m \geq 3$ (the case $m=2$ is known), thus showing that: 
$$
D(n,m)= \left\{
\begin{array} {l} 
3, \,\, \text{if $m=2$}
\\
4, \,\, \text{if $n=2, m=3$}
\\
n(m-2)+1, \,\, \text{otherwise}
\end{array}
\right.
$$
(see Corollary   \ref {cor:trip}, Theorem \ref {thm:main}  and Theorem \ref{thm:sharp}).   
And we shall show that our result is the best possible,  producing  examples, for all  regular homogeneous polynomials $f(x_1,\ldots, x_n)$ of degree $m$, of  a convergent  power series of the form $g=f+ o(d)$, with $d<D(n,m)$, such that  the germ at the origin defined by $g=0$ is not analytically equivalent to the germ defined by $f=0$.

Note that $D(n,m)=m+1$ if and only if $m=2$, $n=2$ and $3\leq m\leq 4$ and $n=m=3$. 

 In Section \ref {sec:qh} we make an extension to the case in which $f$ is a quasihomogeneous polynomial.  After posting this note on math.arxiv, we learned that the above result was already known, and in particular it appears in \cite[Exercise 7.31]{D}. 
 We finally observe that the above result  improves the bound in  \cite[Exercise 1.2.2.7, 2nd ed.]{GLS}, see also \cite[Appendix A, Example 3.5.7, Case 4, 2nd ed.]{GLS}.

\medskip

{\bf Aknowledgements:}  Thanks to  A. Dimca, D. Kerner and J. Koll\'ar for providing useful references. The second  and third authors are members of GNSAGA of INdAM. 
The third author acknowledges funding from the GNSAGA of INdAM and the European Union - NextGenerationEU under the National Recovery and Resilience Plan (PNRR) - Mission 4 Education and research - Component 2 From research to business - Investment 1.1, Prin 2022``Geometry of algebraic structures: moduli, invariants, deformations", DD N. 104, 2/2/2022, proposal code 2022BTA242 - CUP J53D23003720006.

\section{Preliminaries}

\subsection{Finite determinacy} Here we recall some definitions and results from \cite {GLS}.

We will as usual denote by $\hol_n: = \CC\{x_1,\ldots, x_n\}$ the ring of convergent power series $f(x_1,\ldots, x_n)$. 
A  polynomial  $f\in \CC[x_1,\ldots, x_n]$ is said to be \emph{quasihomogeneous} of \emph{isobaric type} $(w_1,\ldots ,w_n; d)$, with $ w_1,\ldots ,w_n \in \NN_+$, if for any monomial $x_1^{a_1}\cdots x_n^{a_n}$ appearing in $f$ one has
$$
w_1a_1+\cdots+ w_na_n=d.
$$
Homogeneous polynomials are quasihomogeneous polynomials of  isobaric type $(1,\ldots ,\\1; d)$. Quasihomogeneous polynomials verify the \emph{Euler relation}
$$
d\cdot f= \sum_{i=1}^n w_ix_i\frac{\partial f}{\partial x_i}.
$$

Consider 
the \emph{Jacobian ideal} $J(f)$ of $f$
$$
J(f) : =  \Big ( \frac {\partial f}{\partial x_1}, \ldots , \frac {\partial f}{\partial x_n}\Big ).
$$

The Euler relation implies that,  if $f$ is quasi-homogeneous, then  the \emph{Milnor Algebra} 
$$M_f : = \CC\{x_1,\ldots, x_n\} /\Big (\frac{\partial f}{\partial x_1}, \dots, \frac{\partial f}{\partial x_n}\Big) = \hol_n / J(f)$$ 
is isomorphic to the \emph{Tyurina Algebra}, i.e., the following quotient of $M_f$: 
$$\sT_f : = \hol_n / ((f) + J(f)) =  \CC\{x_1,\ldots, x_n\} /\Big (f, \ \frac{\partial f}{ \partial x_1}, \dots, \frac{\partial f}{\partial x_n}\Big).$$

\subsection{Saito's theorem} Conversely, an important theorem of 
Kioji  Saito \cite[p. 123]{Sa}  proved that the condition $ M_f  \cong \sT_f$ is also sufficient for $f$ to be quasi-homogeneous.  Since $\sT_f $ is a quotient of $M_f$, this condition is equivalent to the equality of \emph{Milnor number} $ \mu_f $ and \emph{Tyurina number} $\tau_f $, i.e.,
$$ \mu_f : = \dim_{\CC} (M_f) = 
  \dim_{\CC} (\sT_f) = : \tau_f $$ 
  
    Saito's theorem can be reformulated as follows, using the forthcoming Lemma \ref{lem: quasihomogeneous},
the Euler relation for quasihomogeneus polynomials, and the chain rule.

\begin{theorem}\label{thm:saito} Let $f\in \mathfrak m\subset  \hol_n$ such that $f=0$ defines a germ of hypersurface with an isolated singularity at the origin. Then this germ is analytically isomorphic to the germ defined by a quasihomogeneous polynomial if and only if $f\in J(f)$, equivalently,  if and only if  $M_f  \cong \sT_f$.
\end{theorem}

\subsection{Right equivalence} We  now  define an equivalence relation $\sim$ in $ \hol_n$  (called {\it right equivalence} in \cite[Def. 1.2.9, 2nd ed.]{GLS}). Namely, given $f,g\in 
 \hol_n$, we will define $f\sim g$ if and only if there exists an automorphism $\phi$
of $ \hol_n$ such that $$\phi(f)  :=f\circ \phi = g.$$  In general, the right equivalence between $f$ and $g$ implies that  the germs of hypersurfaces defined  by $f=0$ and $g=0$ at the origin are analytically equivalent (in that case one says that $f$ and $g$ are contact equivalent \cite[Rmk. 1.2.9.1 and Def. 1.2.9, 2nd ed.]{GLS}).

\begin{lemma}[Lemma 1.2.13 and Definition 1.2.9 of 2nd ed. \cite{GLS}]\label{lem: quasihomogeneous}
If $f, g \in  \hol_n$ and $f$ is a quasihomogeneous polynomial, then  $f\sim g$ if and only if the germs defined by $f=0$ and $g=0$ at the origin are analytically equivalent.
\end{lemma}
  \color{black}

Given $f\in  \hol_n$ and given a non--negative integer $k$ we define the \emph{$k$--truncation} $f^{(k)}$ of $f$ to be
$$
f^{(k)} : \text{the image of $f$ in} \,  \hol_n /\mathfrak m^{k+1}
$$
where $\mathfrak m$ is the maximal ideal of $ \hol_n$. 

We will say that $f\in  \hol_n$ is    \emph{$k$--determined} (\emph{right $k$--determined} in \cite [Def. 1.2.21, 2nd ed.]{GLS}) if for each $g\in  \hol_n$ with $f^{(k)} = g^{(k)}$ we have $f\sim g$. 

 Notice that, if $f$ is a homogeneous polynomial of degree $m$ that is $m$--determined, then any germ of hypersurface in $\mathbb C^n$ at the origin whose tangent cone  is  defined by $f=0$, is analytically isomorphic to  its  tangent cone. 

The following is the \emph{Finite Determinacy Theorem} from \cite [Thm. 1.2.23, 2nd ed \color{blue}] {GLS}:

\begin{theorem}\label{thm:fdt} Let $f\in \mathfrak m\subset  \hol_n$. Then $f$ is  $k$--determined if
\begin{equation}\label{eq:worp}
\mathfrak m^{k+1} \subseteq \mathfrak m^2 \cdot J(f).
\end{equation}

\end{theorem} 

 When $f$ is homogeneous,  we note that  a necessary condition to apply this theorem is that $f$  be regular, otherwise $ \mathfrak m^2 \cdot J(f)$ cannot contain any power of $\mathfrak m$.

\begin{corollary}\label{determined}
If $f\in \hol_n$ is a homogeneous polynomial of degree $m$, and $ k \geq m$, then 
 $f$ is  $k$--determined if
\begin{equation}\label{eq:worp}
\mathfrak m^{k+1} \subseteq  J(f).
\end{equation}

\end{corollary} 

\begin{proof}
It suffices to show that every homogeneous polynomial $F$ of degree $k+1$ is contained in $ \frak m^2 \cdot J(f)$.
If $F \in J(f)$ and \eqref {eq:worp} holds, then there are $g_1, \dots, g_n\in \hol_n$ such that $ F = \sum_{i=1}^n g_i  \frac {\partial f}{\partial x_i}$,
and passing to the homogeneous part of degree $k+1$, we can take $g_i$ to be a homogeneous 
polynomial of degree $ k+1 - (m-1) = (k-m) + 2 \geq 2$.

Hence $g_i \in  \frak m^2$ for all $1\leq i\leq n$. \end{proof}

Note that \eqref {eq:worp} implies that $f$ is regular. 

Let us denote by 
$$\CC[x_1,\ldots, x_n]_d$$ the vector space of the homogeneous polynomials of degree $d$, which 
has  dimension
$$
N(n,d)= {{d+n-1}\choose {n-1}}.
$$

 As immediate consequences of Theorem \ref {thm:fdt} we have first of all that any regular homogeneous polynomial of degree $2$  is 2--determined.  This is usually derived from the following elementary \emph{Generalized Morse Lemma}  or \emph{Splitting Lemma} (see for instance \cite{Vorlesung}):
  
 \begin{lemma}
  Let $f\in \mathfrak m^2\subset  \hol_n$  with a non--zero homogeneous degree 2 part. Then there are new  coordinates $(z_1, \dots, z_n)$ such that 
 $$ f = z_1 ^2 + \cdots + z_h^2 + F(z_{h+1}, \dots, z_n),$$
 where $h$ is the rank of the Hessian matrix at the origin $0$, and $F\in \mathfrak m^3\subset \mathbb C\{z_{h+1}, \dots, z_n\}$. 
 
 \end{lemma}

\begin{corollary}\label{cor:trip} Any regular homogeneous polynomial of degree $3$ in $\mathbb C\{x_1, x_2\}$ is 3--determined. Hence any germ of curve with an ordinary triple point is analytically isomorphic to its tangent cone.
\end{corollary}

\begin{proof} Consider the multiplication map 
$$
\sM : \mathbb C[x_1,x_2]_2\otimes J(f)_2\longrightarrow \mathbb C[x_1, x_2]_4.
$$
We claim that  this map is surjective. In fact, since the two derivatives of $f$ form a regular sequence, the kernel of $\sM $ is generated by the only Koszul syzygy of the derivatives, hence 
$$
\dim ({\rm Im}( \sM))=3\cdot 2-1=5=\dim (\mathbb C[x_1, x_2]_4).
$$
This implies that   \eqref {eq:worp}  holds for $k=3$.
\end{proof}

More generally one has the following result  due to Tougeron \cite{T},  implying that every isolated singularity is finite determined. 

\begin{corollary}[Cor. 1.2.24 of the second edition of \cite{GLS}  and Sect. 6.3 of \cite{AGLV}]\label{cor:milnor}
Any germ of an isolated singularity with Milnor number $\mu$ is $(\mu+1)$-determined. 
\end{corollary}

\subsection{A combinatorial lemma} In what follows, we will need the following lemma. Given $n,m\geq 2$ integers, set
$$
M(n,m) : =\sum_{h=1}^n (-1)^{h-1} {n \choose h} {{(n-h)(m-1)} \choose {n-1}},
$$
with the convention that ${a \choose b}=0$ if $a<b$. 

\begin{lemma}\label{lem:comb} One has
$$
M(n,m)= {{n(m-1)}\choose {n-1}}.
$$
\end{lemma}

\begin{proof} The proof goes by induction, and by use of  the inclusion--exclusion principle. In fact ${{n(m-1)}\choose {n-1}}$ is the number of choices of a set A of $(n-1)$ objects among $n(m-1)$ given ones. Divide the set of these objects  into $n$ disjoint subsets $X_1,\ldots, X_n$, 
each of cardinality $m-1$. 

Since $A$ has cardinality $n-1$ and we have $n$ sets $X_1,\ldots, X_n$, there is a set $X_j$, for some  $1\leq j\leq n$, 
such that $A$ has empty intersection with $X_j$.

Hence, to  choose $n-1$ objects among these we can choose  $n-1$  objects among the objects in $\cup_{i=1}^nX_i\setminus X_j$ for every $j\in \{1,\ldots, n\}$ and let $j$ vary. This gives the first summand in $M(n,m)$. However we have counted too many  $(n-1)$--tuples.  We have to subtract the $(n-1)$--tuples that are in $\cup_{i=1}^nX_i\setminus (X_j\cup X_h)$, with $j\neq h$. And so on. The assertion follows. 

\end{proof}

\section{Computation of $D(n,m)$} 

\subsection{The main theorem} Our main result is the following:

\begin{theorem}\label{thm:main} Let $f\in \hol_n$ be a regular homogeneous polynomial of degree $m\geq 3$, and exclude the case   $n=2, m=3$ in view of Corollary \ref{cor:trip}. 

Then $f$ is $n(m-2)$--determined. As a consequence we have
\begin{equation}\label{eq:kolp}
D(n,m)= m+1\,{\textrm if} \, n=m=3\, {\textrm and}\, (n,m)=(2,4),
\end{equation}
and $D(n,m)\leq n(m-2)+1$. 
\end{theorem}

\begin{proof}  We  now prove the assertion using  Corollary  \ref {determined},  by observing that  $n(m-2)-m\geq 0$ under our hypotheses. 

It suffices to observe  that the 
 Milnor Algebra is a graded algebra of degree $ \de_f = n (m-2)$, i.e., its homogeneous summands  are zero  
in degree $> n (m-2)$. Indeed, 
 since the derivatives of $f$ form a regular sequence, the Koszul complex is exact and gives a resolution  
 of the (graded) Milnor algebra $M_f: = \hol_n / (J(f))$.
 
  Since we have a complete intersection $\Sigma$  in $\PP^{n-1}$ of type 
 $(m-1)^n$, the dualizing sheaf, by adjunction theory,  should be  $\omega_{\Sigma} \cong \hol_{\Sigma}(n(m-2))$,
 and that this is so was proven by 
Scheja and Storch who showed indeed in \cite {SS} (see also \cite [Thm. 2.3] {vSW}) that the socle of $M_f$, (graded piece of maximal degree)
 is generated by the Hessian $\varphi$ of $f$, such that 
 $$\varphi \ dx_1 \wedge \dots \wedge dx_n = d \Big ( \frac {\partial f}{\partial x_1} \Big ) \wedge \dots \wedge d \Big ( \frac {\partial f}{\partial x_n}\Big ).$$

  Hence the degree of the Milnor algebra is $n(m-2)$ and we are done.
 
 We can also give an elementary proof  checking  that  the multiplication map
$$
\sM : \CC[x_1,\ldots, x_n]_{n(m-2)-m+2}\otimes J(f)_{m-1}\longrightarrow \mathbb C[x_1,\ldots, x_n]_{n(m-2)+1}
$$
is surjective. Again  observe in fact that $n(m-2)-m+2\geq 2$ under our hypotheses. 

By the exactness of the Koszul complex, we infer that 
the kernel of $\sM  $ has dimension equal to 
$$
\sum_{h=2}^n (-1)^{h} {n \choose h} {{(n-h)(m-1)} \choose {n-1}},
$$
thus the image of $\sM  $ has dimension $M(n,m)$. So the assertion follows from  
Lemma \ref {lem:comb}. 

\end{proof}

\subsection{Sharpness}

The result in Theorem \ref {thm:main} is sharp, as the following theorem shows:

\begin{theorem}\label{thm:sharp} 
Assume that $m\geq 3$, and $n\geq 4$ if $m=3$, $n\geq 3$ if $m=4$. Then  the \emph{Fermat polynomial} 
$$
f_{n,m}(x_1,\ldots, x_n):= x_1^m+\cdots + x_n^m
$$
is not $n(m-2)-1$--determined.  In particular, 
\begin{equation}\label{eq:kolp2}
D(n,m)=n(m-2)+1.
\end{equation}

\end{theorem}

\begin{proof}  We argue by contradiction. Suppose that $f_{n,m}$ is ($n(m-2)-1$)--determined. Observe that, under our assumptions on $(n,m)$, one has that  $n(m-2)>m$.  Then consider the polynomial
$$
g_{n,m}(x_1,\ldots, x_n,t):=f_{n,m}+t x_1^{m-2} \cdots x_n^{m-2}.
$$
Since $f_{n,m}$ is $(n(m-2)-1)$--determined, then, for all $t\neq 0$, the germ of the hypersurface $g_{n,m}(x_1,\ldots, x_n,t)=0$ is analytically equivalent to the germ of $f_{n,m}(x_1,\ldots, x_n)\\
=0$. 

Therefore, by Saito's Theorem \ref {thm:saito}, $g_{n,m}(x_1,\ldots, x_n,t)$ belongs to the Jacobian ideal $J_t$ of $g_{n,m}(x_1,\ldots, x_n,t)$ as polynomial in $x_1,\ldots, x_n$.  Now $J_t$  is generated by
$$
\frac 1m \cdot \frac {\partial g_{n,m}}{\partial x_i}=x_i^{m-1}+\frac {t(m-2)}m 
x_1^{m-2} \cdots x_i^{m-3} \cdots x_n^{m-2}, \, \text{for}\, 1\leq i\leq n.
$$
Therefore the  polynomial
$$
\frac 1m \cdot \sum_{i=1}^n x_i\frac {\partial g_{n,m}}{\partial x_i}=f_{n,m}+\frac {t(m-2)n}m x_1^{m-2} \cdots x_n^{m-2}
$$
belongs to $J_t$. Then belongs to $J_t$ also the polynomial
$$
\frac 1m \cdot \sum_{i=1}^n x_i\frac {\partial g_{n,m}}{\partial x_i}-g_{n,m}=
t\Big (\frac {n(m-2)}m-1\Big)x_1^{m-2} \cdots x_n^{m-2}
$$
hence $x_1^{m-2} \cdots x_n^{m-2}$ belongs to $J_t$ for all $t\neq 0$, because 
$nm-2n-m\neq 0$ under our hypotheses.  Then $x_1^{m-2} \cdots x_n^{m-2}$ belongs to $\lim_{t\to 0} J_t=J(f_{n,m})=(x_1^{m-1}, \ldots, x_n^{m-1})$, and this  is a contradiction. 
\end{proof}

\begin{remark}
The expert reader will now be easily convinced that the same counterexample holds for every 
regular homogeneous polynomial $f$.

In fact,  by the cited  result of Scheja and Storch,  the Milnor algebra $M_f$ is a graded Gorenstein algebra whose socle, i.e., its graded piece
of maximal degree $\de_f = n (m-2)$, has dimension 1 and is generated by a the polynomial  $\varphi : = {\rm Hess}(f)$. 

If we take $ g_t :=f  + t \varphi$, then for general $t$ we have that $g_t$ is not analytically
equivalent to $f$.
\end{remark}

\section{Generalization to the quasi-homogeneous case}\label{sec:qh}

If $f$ is quasi-homogeneous of  isobaric type $(w_1,\ldots ,w_n; m)$,  with $ w_1,\ldots ,w_n \in \NN_+$,
then we say again that $f$ is \emph{regular} if the partial derivatives form a regular sequence.
Then  the Koszul complex is exact and 
the Milnor algebra $M_f$ is a Gorenstein algebra of  
 weighted degree  $\de_f$,
which can be calculated by the cited result of Scheja and Storch, as the degree of the Hessian $\varphi: = {\rm Hess}(f)$,
which equals 
 $$ \de_f = nm-2 \sum_{j=1}^n w_j.$$

 With  the same proof, we have  then the following:

\begin{theorem}
Let $f$ be quasi-homogeneous polynomial of  isobaric type $(w_1,\ldots ,w_n; m)$.
Then $f$ is $ (nm-2 \sum_{j=1}^n w_j)$-determined.
\end{theorem}

{}
\end{document}